\documentclass[reqno,12pt]{amsart}

\usepackage{amsfonts,color}
\usepackage{amssymb,amsmath}
\usepackage[all]{xy}

 \newtheorem{theorem}{Theorem}
 \newtheorem{corollary}[theorem]{Corollary}
 \newtheorem{lemma}[theorem]{Lemma}
 \newtheorem{proposition}[theorem]{Proposition}
 \theoremstyle{definition}
 
 \theoremstyle{remark}
 \newtheorem{remark}[theorem]{Remark}
 \newtheorem{example}[theorem]{Example}

 \numberwithin{equation}{section}
 \numberwithin{theorem}{section}

\begin{document}

\title[Strong extensions for $q$-summing operators]
{Strong extensions for $q$-summing operators acting in
$p$-convex Banach function spaces for $1 \le p \le q$}

\author[O.\ Delgado]{O.\ Delgado}
\address{Departamento de Matem\'atica Aplicada I, E.\ T.\ S.\ de Ingenier\'ia
de Edificaci\'on, Universidad de Sevilla, Avenida de Reina Mercedes,
4 A,  Sevilla 41012, Spain}
\email{\textcolor[rgb]{0.00,0.00,0.84}{olvido@us.es}}

\author[E.\ A.\ S\'{a}nchez P\'{e}rez]{E.\ A.\ S\'{a}nchez P\'{e}rez}
\address{Instituto Universitario de Matem\'atica Pura y Aplicada,
Universitat Polit\`ecnica de Val\`encia, Camino de Vera s/n, 46022
Valencia, Spain.}
\email{\textcolor[rgb]{0.00,0.00,0.84}{easancpe@mat.upv.es}}



\thanks{The first author gratefully acknowledge the support of the Ministerio de Econom\'{\i}a y Competitividad
(project \#MTM2012-36732-C03-03) and the Junta de Andaluc\'{\i}a
(projects FQM-262 and FQM-7276), Spain.}
\thanks{The second author acknowledges with thanks the support of the Ministerio de Econom\'{\i}a y Competitividad
(project \#MTM2012-36740-C02-02), Spain.}

\date{\today}

\maketitle


\begin{abstract}
Let $1\le p\le q<\infty$ and let $X$ be a $p$-convex Banach function
space over a $\sigma$-finite measure $\mu$. We combine the structure of the spaces $L^p(\mu)$ and $L^q(\xi)$  for constructing
the new space  $S_{X_p}^{\,q}(\xi)$, where $\xi$ is a probability Radon
measure on a certain compact set associated to $X$.  We show some of its properties, and the relevant fact that
every $q$-summing operator $T$ defined on $X$ can be continuously (strongly)
extended to $S_{X_p}^{\,q}(\xi)$. This result turns out to be a
mixture of the Pietsch and Maurey-Rosenthal factorization theorems, which provide (strong) factorizations for $q$-summing operators through $L^q$-spaces when  $1 \le q \le p$. Thus, our result  completes the picture, showing what happens in the  complementary case  $1\le p\le q$, opening the door to the study of the multilinear versions of $q$-summing operators also in these cases.
\end{abstract}



\section{Introduction}

Fix $1\le p\le q<\infty$ and let $T\colon X\to E$ be a Banach space
valued linear operator defined on a saturated order semi-continuous
Banach function space $X$ related to a $\sigma$-finite measure
$\mu$. In this paper we prove an extension theorem for $T$ in the
case when $T$ is $q$-summing and $X$ is $p$-convex. In order to do
this, we first define and analyze a new class of Banach function
spaces denoted by $S_{X_p}^{\,q}(\xi)$ which have some good
properties, mainly order continuity and p-convexity. The space
$S_{X_p}^{\,q}(\xi)$ is constructed by using the spaces $L^p(\mu)$
and $L^q(\xi)$, where $\xi$ is a finite positive Radon measure on a
certain compact set associated to $X$.

Corollary \ref{COR: q-summing-extension} states the desired
extension for $T$. Namely, if $T$ is $q$-summing and $X$ is
$p$-convex then $T$ can be strongly extended continuously to a space of the
type $S_{X_p}^{\,q}(\xi)$. Here we use the term ``strongly" for this extension to remark that the map carrying $X$ into $S_{X_p}^{\,q}(\xi)$ is actually injective; as the reader will notice (Proposition \ref{PROP: SXpq(xi)-space}), this is one of the goals of our result. In order to develop our arguments, we introduce
a new geometric tool which we call the family of $p$-strongly
$q$-concave operators. The inclusion of $X$ into
$S_{X_p}^{\,q}(\xi)$ turns out to belong to this family, in
particular, it is $q$-concave.

If $T$ is $q$-summing then it is $p$-strongly $q$-concave
(Proposition \ref{PROP: q-Summing}). Actually, in Theorem \ref{THM:
SXpqExtension} we show that in the case when $X$ is $p$-convex, $T$
can be continuously extended to a space $S_{X_p}^{\,q}(\xi)$ if and
only if $T$ is $p$-strongly $q$-concave. This result can be
understood as an extension of some well-known relevant
factorizations of the operator theory:

\begin{itemize}\setlength{\leftskip}{-3ex}
\item[(I)] Maurey-Rosenthal factorization theorem: If $T$ is
$q$-concave and $X$ is $q$-convex and order continuous, then $T$ can
be extended to a weighted $L^q$-space related to $\mu$, see for
instance \cite[Corollary 5]{defant}. Several generalizations and
applications of the ideas behind this fundamental factorization
theorem have been recently obtained, see
\cite{calabuig-delgado-sanchezperez,
calabuig-rodriguez-sanchezperez,defant-sanchezperez,delgado-sanchezperez,sanchezperez}.

\item[(II)] Pietsch factorization theorem:
If $T$ is $q$-summing then it factors through a closed subspace
of $L^q(\xi)$, where $\xi$ is a probability Radon measure on a
certain compact set associated to $X$, see for instance
\cite[Theorem 2.13]{diestel-jarchow-tonge}.
\end{itemize}

In Theorem \ref{THM: SXpqExtension}, the extreme case $p=q$ gives a
Maurey-Rosenthal type factorization, while the other extreme case
$p=1$ gives a Pietsch type factorization. We must say also that our generalization will allow to face the problem of the factorization of several $p$-summing type of multilinear operators from products of Banach function spaces ---a topic of current interest---, since it allows to understand factorization of $q$-summing operators from $p$-convex function lattices from a unified point of view not depending on the order relation between $p$ and $q$.

As a consequence of Theorem \ref{THM: SXpqExtension}, we also prove
a kind of Kakutani representation theorem (see for instance
\cite[Theorem 1.b.2]{lindenstrauss-tzafriri}) through the spaces
$S_{X_p}^{\,q}(\xi)$ for $p$-convex Banach function spaces which are
$p$-strongly $q$-concave (Corollary \ref{COR: i-isomorphism}).


\section{Preliminaries}

Let $(\Omega,\Sigma,\mu)$ be a $\sigma$-finite measure space and
denote by $L^0(\mu)$ the space of all measurable real functions on
$\Omega$, where functions which are equal $\mu$-a.e.\ are
identified. By a \emph{Banach function space} (briefly B.f.s.) we
mean a Banach space $X\subset L^0(\mu)$ with norm
$\Vert\cdot\Vert_X$, such that if $f\in L^0(\mu)$, $g\in X$ and
$|f|\le|g|$ $\mu$-a.e.\ then $f\in X$ and $\Vert f\Vert_X\le\Vert
g\Vert_X$. In particular, $X$ is a Banach lattice with the
$\mu$-a.e.\ pointwise order, in which the convergence in norm of a
sequence implies the convergence $\mu$-a.e.\ for some subsequence. A
B.f.s.\ $X$ is said to be \emph{saturated} if there exists no
$A\in\Sigma$ with $\mu(A)>0$ such that $f\chi_A=0$ $\mu$-a.e.\ for
all $f\in X$, or equivalently, if X has a \emph{weak unit} (i.e.\
$g\in X$ such that $g>0$ $\mu$-a.e.).

\begin{lemma}\label{LEM: saturatedBfs}
Let $X$ be a saturated B.f.s. For every $f\in L^0(\mu)$, there
exists $(f_n)_{n\ge1}\subset X$ such that $0\le f_n\uparrow |f|$
$\mu$-a.e.
\end{lemma}

\begin{proof}
Consider a weak unit $g\in X$ and take $g_n=ng/(1+ng)$. Note that
$0< g_n< ng$ $\mu$-a.e., so $g_n$ is a weak unit in $X$. Moreover,
$(g_n)_{n\ge1}$ increases $\mu$-a.e.\ to the constant function equal
to $1$. Now, take $f_n=g_n|f|\chi_{\{\omega\in\Omega:\,|f|\le n\}}$.
Since $0\le f_n\le ng_n$ $\mu$-a.e., we have that $f_n\in X$, and
$f_n\uparrow |f|$ $\mu$-a.e.
\end{proof}

The \emph{K\"{o}the dual} of a B.f.s.\ $X$ is the space $X'$ given
by the functions $h\in L^0(\mu)$ such that $\int|hf|\,d\mu<\infty$
for all $f\in X$. If $X$ is saturated then $X'$ is a saturated
B.f.s.\ with norm $\Vert h\Vert_{X'}=\sup_{f\in B_X}\int|hf|\,d\mu$
for $h\in X'$. Here, as usual, $B_X$ denotes the closed unit ball of
$X$. Each function $h\in X'$ defines a functional $\zeta(h)$ on $X$
by $\langle\zeta(h),f\rangle=\int hf\,d\mu$ for all $f\in X$. In
fact, $X'$ is isometrically order isomorphic (via $\zeta$) to a
closed subspace of the topological dual $X^*$ of $X$.

From now and on, a B.f.s.\ $X$ will be assumed to be saturated. If
for every $f,f_n\in X$ such that $0\le f_n\uparrow f$ $\mu$-a.e.\ it
follows that $\Vert f_n\Vert_X\uparrow \Vert f\Vert_X$, then $X$ is
said to be \emph{order semi-continuous}. This is equivalent to
$\zeta(X')$ being a \emph{norming subspace} of $X^*$, i.e.\ $\Vert
f\Vert_X=\sup_{h\in B_{X'}}\int |fh|\,d\mu$ for all $f\in X$. A
B.f.s.\ $X$ is \emph{order continuous} if for every $f,f_n\in X$
such that $0\le f_n\uparrow f$ $\mu$-a.e., it follows that $f_n\to
f$ in norm. In this case, $X'$ can be identified with $X^*$.

For general issues related to B.f.s.'\ see
\cite{lindenstrauss-tzafriri}, \cite{okada-ricker-sanchezperez} and
\cite[Ch.\,15]{zaanen} considering the function norm $\rho$ defined
as $\rho(f)=\Vert f\Vert_X$ if $f\in X$ and $\rho(f)=\infty$ in
other case.

Let $1\le p<\infty$. A B.f.s.\ $X$ is said to be  \emph{$p$-convex}
if there exists a constant $C>0$ such that
$$
\Big\Vert\Big(\sum_{i=1}^n |f_i|^p\Big)^{1/p}\,\Big\Vert_X \le C
\Big(\sum_{i=1}^n\Vert f_i\Vert_X^p\Big)^{1/p}
$$
for every finite subset $(f_i)_{i=1}^n\subset X$. In this case,
$M^p(X)$ will denote the smallest constant $C$ satisfying the above
inequality. Note that $M^p(X)\ge1$. A relevant fact is that every
$p$-convex B.f.s.\ $X$ has an equivalent norm for which $X$ is
$p$-convex with constant $M^p(X)=1$, see \cite[Proposition
1.d.8]{lindenstrauss-tzafriri}.

The \emph{$p$-th power} of a B.f.s.\ $X$ is the space defined as
$$
X_p=\{f\in L^0(\mu): |f|^{1/p}\in X\},
$$
endowed with the quasi-norm $\Vert f \Vert_{X_p}= \Vert
\,|f|^{1/p}\,\Vert_X^p$, for $f\in X_p$. Note that $X_p$ is always
complete, see the proof of \cite[Proposition
2.22]{okada-ricker-sanchezperez}. If $X$ is $p$-convex with constant
$M^p(X)=1$, from \cite[Lemma 3]{defant}, $\Vert \cdot\Vert_{X_p}$ is
a norm and so $X_p$ is a B.f.s. Note that $X_p$ is saturated if and
only if $X$ is so. The same holds for the properties of being order
continuous and order semi-continuous.


\section{The space $S_{X_p}^{\,q}(\xi)$}

Let $1\le p\le q<\infty$ and let $X$ be a saturated $p$-convex
B.f.s. We can assume without loss of generality that the
$p$-convexity constant $M^p(X)$ is equal to $1$. Then, $X_p$ and
$(X_p)'$ are saturated B.f.s.'. Consider the topology
$\sigma\big((X_p)',X_p\big)$ on $(X_p)'$ defined by the elements of
$X_p$. Note that the subset $B_{(X_p)'}^+$ of all positive elements
of the closed unit ball of $(X_p)'$ is compact for this topology.

Let $\xi$ be a finite positive Radon measure on $B_{(X_p)'}^+$. For
$f\in L^0(\mu)$, consider the map $\phi_f\colon
B_{(X_p)'}^+\to[0,\infty]$ defined by
$$
\phi_f(h)=\Big(\int_\Omega|f(\omega)|^ph(\omega)\,d\mu(\omega)\Big)^{q/p}
$$
for all $h\in B_{(X_p)'}^+$. In the case when $f\in X$, since
$|f|^p\in X_p$, it follows that $\phi_f$ is continuous and so
measurable. For a general $f\in L^0(\mu)$, by Lemma \ref{LEM:
saturatedBfs} we can take a sequence $(f_n)_{n\ge1}\subset X$ such
that $0\le f_n\uparrow|f|$ $\mu$-a.e. Applying monotone convergence
theorem, we have that $\phi_{f_n}\uparrow\phi_f$ pointwise and so
$\phi_f$ is measurable. Then, we can consider the integral
$\int_{B_{(X_p)'}^+}\phi_f(h)d\xi(h)\in[0,\infty]$ and define the
following space:
$$
S_{X_p}^{\,q}(\xi)=\left\{f\in L^0(\mu):\,
\int_{B_{(X_p)'}^+}\Big(\int_\Omega|f(\omega)|^ph(\omega)\,d\mu(\omega)\Big)^{q/p}d\xi(h)<\infty\right\}.
$$
Let us endow $S_{X_p}^{\,q}(\xi)$ with the seminorm
\begin{eqnarray*}
\Vert f\Vert_{S_{X_p}^{\,q}(\xi)} & = &
\left(\int_{B_{(X_p)'}^+}\Big(\int_\Omega|f(\omega)|^ph(\omega)\,d\mu(\omega)\Big)^{q/p}d\xi(h)\right)^{1/q}
\\ & = & \Big\Vert \,h\to\big\Vert f|h|^{1/p}\,\big\Vert_{L^p(\mu)}\,\Big\Vert_{L^q(\xi)}.
\end{eqnarray*}
In general, $\Vert \cdot\Vert_{S_{X_p}^{\,q}(\xi)}$ is not a norm.
For instance, if $\xi$ is the Dirac measure at some $h_0\in
B_{(X_p)'}^+$ such that $A=\{\omega\in\Omega:\, h_0(\omega)=0\}$
satisfies $\mu(A)>0$, taking $f=g\chi_A\in X$ with $g$ being a weak
unit of $X$, we have that
$$
\Vert
f\Vert_{S_{X_p}^{\,q}(\xi)}=\Big(\int_A|g(\omega)|^ph_0(\omega)\,d\mu(\omega)\Big)^{1/p}=0
$$
and
$$
\mu(\{\omega\in\Omega:\, f(\omega)\not=0\})=
\mu(A\cap\{\omega\in\Omega:\, g(\omega)\not=0\})=\mu(A)>0.
$$

\begin{proposition}\label{PROP: SXpq(xi)-space}
If the Radon measure $\xi$ satisfies
\begin{equation}\label{EQ: xiProperty}
\int_{B_{(X_p)'}^+}\Big(\int_Ah(\omega)\,d\mu(\omega)\Big)^{q/p}\,d\xi(h)=0
\ \ \Rightarrow \ \ \mu(A)=0
\end{equation}
then, $S_{X_p}^{\,q}(\xi)$ is a saturated B.f.s. Moreover,
$S_{X_p}^{\,q}(\xi)$ is order continuous, $p$-convex (with constant
$1$) and $X\subset S_{X_p}^{\,q}(\xi)$ continuously.
\end{proposition}

\begin{proof}
It is clear that if $f\in L^0(\mu)$, $g\in S_{X_p}^{\,q}(\xi)$ and
$|f|\le|g|$ $\mu$-a.e.\ then $f\in S_{X_p}^{\,q}(\xi)$ and $\Vert
f\Vert_{S_{X_p}^{\,q}(\xi)}\le\Vert g\Vert_{S_{X_p}^{\,q}(\xi)}$.
Let us see that $\Vert \cdot\Vert_{S_{X_p}^{\,q}(\xi)}$ is a norm.
Suppose that $\Vert f\Vert_{S_{X_p}^{\,q}(\xi)}=0$ and set
$A_n=\{\omega\in\Omega:\, |f(\omega)|>\frac{1}{n}\}$ for every
$n\ge1$. Since $\chi_{A_n}\le n|f|$ and
$$
\int_{B_{(X_p)'}^+}\Big(\int_{A_n}h(\omega)\,d\mu(\omega)\Big)^{q/p}\,d\xi(h)=\big\Vert
\chi_{A_n}\big\Vert_{S_{X_p}^{\,q}(\xi)}^q\le n^q\Vert
f\Vert_{S_{X_p}^{\,q}(\xi)}^q=0,
$$
from \eqref{EQ: xiProperty} we have that $\mu(A_n)=0$ and so
$$
\mu(\{\omega\in\Omega:\,
f(\omega)\not=0\})=\lim_{n\to\infty}\mu(A_n)=0.
$$

Now we will see that $S_{X_p}^{\,q}(\xi)$ is complete by showing
that $\sum_{n\ge1}f_n\in S_{X_p}^{\,q}(\xi)$ whenever
$(f_n)_{n\ge1}\subset S_{X_p}^{\,q}(\xi)$ with $C=\sum \Vert
f_n\Vert_{ S_{X_p}^{\,q}(\xi)}<\infty$. First let us prove that
$\sum_{n\ge1}|f_n|<\infty$ $\mu$-a.e. For every $N,n\ge1$, taking
$A_n^N=\{\omega\in\Omega:\, \sum_{j=1}^n|f_j(\omega)|>N\}$, since
$\chi_{A_n^N}\le \frac{1}{N}\sum_{j=1}^n|f_j|$, we have that
\begin{eqnarray*}
\int_{B_{(X_p)'}^+}\Big(\int_{A_n^N}h(\omega)\,d\mu(\omega)\Big)^{q/p}\,d\xi(h)
& = & \Vert\chi_{A_n^N}\Vert_{S_{X_p}^{\,q}(\xi)}^q  \\ & \le &
\frac{1}{N^q}\,\Big\Vert\sum_{j=1}^n|f_j|\,\Big\Vert_{S_{X_p}^{\,q}(\xi)}^q
\le\frac{C^q}{N^q}.
\end{eqnarray*}
Note that, for $N$ fixed, $(A_n^N)_{n\ge1}$ increases. Taking limit
as $n\to\infty$ and applying twice the monotone convergence theorem,
it follows that
$$
\int_{B_{(X_p)'}^+}\Big(\int_{\cup_{n\ge1}A_n^N}h(\omega)\,d\mu(\omega)\Big)^{q/p}\,d\xi(h)\le
\frac{C^q}{N^q}.
$$
Then,
$$
\int_{B_{(X_p)'}^+}\Big(\int_{\cap_{N\ge1}\cup_{n\ge1}A_n^N}h(\omega)\,d\mu(\omega)\Big)^{q/p}\,d\xi(h)
\le\lim_{N\to\infty}\frac{C^q}{N^q}=0,
$$
and so, from \eqref{EQ: xiProperty},
$$
\mu\Big(\Big\{\omega\in\Omega:\,\sum_{n\ge1}|f_n(\omega)|=\infty\Big\}\Big)=
\mu\Big(\bigcap_{N\ge1}\bigcup_{n\ge1}A_n^N\Big)=0.
$$
Hence, $\sum_{n\ge1}f_n\in L^0(\mu)$. Again applying the monotone
convergence theorem, it follows that
\begin{eqnarray*}
\int_{B_{(X_p)'}^+}\Big(\int_\Omega\Big|\sum_{n\ge1}f_n(\omega)\Big|^ph(\omega)\,d\mu(\omega)\Big)^{q/p}d\xi(h)
& \le & \\
\int_{B_{(X_p)'}^+}\Big(\int_\Omega\big(\sum_{n\ge1}|f_n(\omega)|\big)^ph(\omega)\,d\mu(\omega)\Big)^{q/p}d\xi(h)
& = & \\
\lim_{n\to\infty}\int_{B_{(X_p)'}^+}\Big(\int_\Omega\big(\sum_{j=1}^n|f_j(\omega)|\big)^ph(\omega)\,d\mu(\omega)\Big)^{q/p}d\xi(h)
& = & \\ \lim_{n\to\infty}\Big\Vert
\sum_{j=1}^n|f_j|\Big\Vert_{S_{X_p}^{\,q}(\xi)}^q & \le & C^q
\end{eqnarray*}
and thus $\sum_{n\ge1}f_n\in S_{X_p}^{\,q}(\xi)$.

Note that if $f\in X$, for every $h\in B_{(X_p)'}^+$ we have that
$$
\int_\Omega|f(\omega)|^ph(\omega)\,d\mu(\omega)\le\Vert\,|f|^p\,\Vert_{X_p}\Vert
h\Vert_{(X_p)'}\le\Vert f\Vert_X^p
$$
and so
$$
\int_{B_{(X_p)'}^+}\Big(\int_\Omega|f(\omega)|^ph(\omega)\,d\mu(\omega)\Big)^{q/p}d\xi(h)
\le\Vert f\Vert_X^q\,\xi\big(B_{(X_p)'}^+\big).
$$
Then, $X\subset S_{X_p}^{\,q}(\xi)$ and $\Vert
f\Vert_{S_{X_p}^{\,q}(\xi)}\le\xi\big(B_{(X_p)'}^+\big)^{1/q}\,\Vert
f\Vert_X$ for all $f\in X$. In particular, $S_{X_p}^{\,q}(\xi)$ is
saturated, as a weak unit in $X$ is a weak unit in
$S_{X_p}^{\,q}(\xi)$.

Let us show that $S_{X_p}^{\,q}(\xi)$ is order continuous. Consider
$f,f_n\in S_{X_p}^{\,q}(\xi)$ such that $0\le f_n\uparrow f$
$\mu$-a.e. Note that, since
$$
\int_{B_{(X_p)'}^+}\Big(\int_\Omega|f(\omega)|^ph(\omega)\,d\mu(\omega)\Big)^{q/p}d\xi(h)<\infty,
$$
there exists a $\xi$-measurable set $B$ with
$\xi(B_{(X_p)'}^+\backslash B)=0$ such that
$\int_\Omega|f(\omega)|^ph(\omega)\,d\mu(\omega)<\infty\,$ for all
$h\in B$. Fixed $h\in B$, we have that $|f-f_n|^ph\downarrow0$
$\mu$-a.e.\ and $|f-f_n|^ph\le |f|^ph$ $\mu$-a.e. Then, applying the
dominated convergence theorem,
$\int_\Omega|f(\omega)-f_n(\omega)|^ph(\omega)\,d\mu(\omega)\downarrow0$.
Consider the measurable functions $\phi,\phi_n\colon
B_{(X_p)'}^+\to[0,\infty]$ given by
\begin{eqnarray*}
\phi(h) & = &
\Big(\int_\Omega|f(\omega)|^ph(\omega)\,d\mu(\omega)\Big)^{q/p} \\
\phi_n(h) & = &
\Big(\int_\Omega|f(\omega)-f_n(\omega)|^ph(\omega)\,d\mu(\omega)\Big)^{q/p}
\end{eqnarray*}
for all $h\in B_{(X_p)'}^+$. It follows that $\phi_n\downarrow0$
$\xi$-a.e.\ and $\phi_n\le \phi$ $\xi$-a.e. Again by the dominated
convergence theorem, we obtain
$$
\Vert f-f_n\Vert_{S_{X_p}^{\,q}(\xi)}^q=
\int_{B_{(X_p)'}^+}\phi_n(h)d\xi(h)\downarrow0.
$$

Finally, let us see that $S_{X_p}^{\,q}(\xi)$ is $p$-convex. Fix
$(f_i)_{i=1}^n\subset S_{X_p}^{\,q}(\xi)$ and consider the
measurable functions $\phi_i\colon B_{(X_p)'}^+\to[0,\infty]$ (for
$1\le i\le n$) defined by
$$
\phi_i(h)=\int_\Omega|f_i(\omega)|^ph(\omega)\,d\mu(\omega).
$$
for all $h\in B_{(X_p)'}^+$. Then,
\begin{eqnarray*}
\Big\Vert\Big(\sum_{i=1}^n
|f_i|^p\Big)^{1/p}\,\Big\Vert_{S_{X_p}^{\,q}(\xi)}^q & = &
\int_{B_{(X_p)'}^+}\Big(\int_\Omega\sum_{i=1}^n|f_i(\omega)|^ph(\omega)\,d\mu(\omega)\Big)^{q/p}d\xi(h)
\\ & = &
\int_{B_{(X_p)'}^+}\Big(\sum_{i=1}^n\phi_i(h)\Big)^{q/p}d\xi(h)
\\ & \le & \Big(\sum_{i=1}^n\Vert
\phi_i\Vert_{L^{q/p}(\xi)}\Big)^{q/p}.
\end{eqnarray*}
Since $\Vert \phi_i\Vert_{L^{q/p}(\xi)}=\Vert
f_i\Vert_{S_{X_p}^{\,q}(\xi)}^p$ for all $1\le i\le n$, we have that
$$
\Big\Vert\Big(\sum_{i=1}^n
|f_i|^p\Big)^{1/p}\,\Big\Vert_{S_{X_p}^{\,q}(\xi)}
\le\Big(\sum_{i=1}^n\Vert
f_i\Vert_{S_{X_p}^{\,q}(\xi)}^p\Big)^{1/p}.
$$
\end{proof}

\begin{example}
Take a weak unit $g\in(X_p)'$ and consider the Radon measure $\xi$
as the Dirac measure at $g$. If $A\in\Sigma$ is such that
$$
0=
\int_{B_{(X_p)'}^+}\Big(\int_Ah(\omega)\,d\mu(\omega)\Big)^{q/p}\,d\xi(h)
=\Big(\int_Ag(\omega)\,d\mu(\omega)\Big)^{q/p}
$$
then, $g\chi_A=0$ $\mu$-a.e.\ and so, since $g>0$ $\mu$-a.e.,
$\mu(A)=0$. That is, $\xi$ satisfies \eqref{EQ: xiProperty}. In this
case, $S_{X_p}^{\,q}(\xi)=L^p(gd\mu)$ with equal norms, as
$$
\int_{B_{(X_p)'}^+}\Big(\int_\Omega
|f(\omega)|^ph(\omega)\,d\mu(\omega)\Big)^{q/p}\,d\xi(h)=
\Big(\int_\Omega|f(\omega)|^pg(\omega)\,d\mu(\omega)\Big)^{q/p}
$$
for all $f\in L^0(\mu)$.
\end{example}

\begin{example}
Write $\Omega=\cup_{n\ge1}\Omega_n$ with $(\Omega_n)_{n\ge1}$ being
a disjoint sequence of measurable sets and take a sequence of
strictly positive elements $(\alpha_n)_{n\ge1}\in\ell^1$. Let us
consider the Radon measure
$\xi=\sum_{n\ge1}\alpha_n\delta_{g\chi_{\Omega_n}}$ on
$B_{(X_p)'}^+$, where $\delta_{g\chi_{\Omega_n}}$ is the Dirac
measure at $g\chi_{\Omega_n}$ with $g\in(X_p)'$ being a weak unit.
Note that for every positive function $\phi\in L^0(\xi)$, it follows
that
$\int_{B_{(X_p)'}^+}\phi\,d\xi=\sum_{n\ge1}\alpha_n\phi(g\chi_{\Omega_n})$.
If $A\in\Sigma$ is such that
$$
0=\int_{B_{(X_p)'}^+}\Big(\int_Ah(\omega)\,d\mu(\omega)\Big)^{q/p}\,d\xi(h)
=\sum_{n\ge1}\alpha_n\Big(\int_{A\cap\Omega_n}g(\omega)\,d\mu(\omega)\Big)^{q/p}
$$
then, $\int_{A\cap\Omega_n}g(\omega)\,d\mu(\omega)=0$ for all
$n\ge1$. Hence,
$$
\int_Ag(\omega)\,d\mu(\omega)=\sum_{n\ge1}\int_{A\cap\Omega_n}g(\omega)\,d\mu(\omega)=0
$$
and so $g\chi_A=0$ $\mu$-a.e., from which $\mu(A)=0$. That is, $\xi$
satisfies \eqref{EQ: xiProperty}. For every $f\in L^0(\mu)$ we have
that
\begin{eqnarray*}
\int_{B_{(X_p)'}^+}\Big(\int_\Omega
|f(\omega)|^ph(\omega)\,d\mu(\omega)\Big)^{q/p}\,d\xi(h)= \\
\sum_{n\ge1}\alpha_n\Big(\int_{\Omega_n}|f(\omega)|^pg(\omega)\,d\mu(\omega)\Big)^{q/p}.
\end{eqnarray*}
Then, the B.f.s.\ $S_{X_p}^{\,q}(\xi)$ can be described as the space
of functions $f\in \cap_{n\ge1}L^p(g\chi_{\Omega_n}d\mu)$ such that
$\big(\alpha_n^{1/q}\Vert
f\Vert_{L^p(g\chi_{\Omega_n}d\mu)}\big)_{n\ge1}\in\ell^q$. Moreover,
$\Vert f\Vert_{S_{X_p}^{\,q}(\xi)}=\Big(\sum_{n\ge1}\alpha_n\,\Vert
f\Vert_{L^p(g\chi_{\Omega_n}d\mu)}^q\Big)^{1/q}$ for all $f\in
S_{X_p}^{\,q}(\xi)$.
\end{example}


\section{$p$-strongly $q$-concave operators}

Let $1\le p\le q<\infty$ and let $T\colon X\to E$ be a linear
operator from a saturated B.f.s.\ $X$ into a Banach space $E$.
Recall that $T$ is said to be \emph{$q$-concave} if there exists a
constant $C>0$ such that
$$
\Big(\sum_{i=1}^n\Vert T(f_i) \Vert_E^q\Big)^{1/q} \le C
\Big\Vert\Big(\sum_{i=1}^n|f_i|^q\Big)^{1/q}\,\Big\Vert_X
$$
for every finite subset $(f_i)_{i=1}^n\subset X$. The smallest
possible value of $C$ will be denoted by $M_q(T)$. For issues
related to $q$-concavity see for instance
\cite[Ch.\,1.d]{lindenstrauss-tzafriri}. We introduce a little
stronger notion than $q$-concavity: $T$ will be called
\emph{$p$-strongly $q$-concave} if there exists $C>0$ such that
$$
\Big(\sum_{i=1}^n\Vert T(f_i)\Vert_E^q\Big)^{1/q}\le C
\sup_{(\beta_i)_{i\ge1} \in B_{\ell^r}}\Big\Vert\Big(\sum_{i=1}^n
|\beta_if_i|^p\Big)^{1/p}\,\Big\Vert_X
$$
for every finite subset $(f_i)_{i=1}^n\subset X$, where $1<
r\le\infty$ is such that $\frac{1}{r}=\frac{1}{p}-\frac{1}{q}$. In
this case, $M_{p,q}(T)$ will denote the smallest constant $C$
satisfying the above inequality. Noting that $\frac{r}{p}$ and
$\frac{q}{p}$ are conjugate exponents, it is clear that every
$p$-strongly $q$-concave operator is $q$-concave and so continuous,
and moreover $\Vert T\Vert\le M_q(T)\le M_{p,q}(T)$. As usual, we
will say that $X$ is \emph{$p$-strongly $q$-concave} if the identity
map $I\colon X\to X$ is so, and in this case, we denote
$M_{p,q}(X)=M_{p,q}(I)$.

Our goal is to get a continuous extension of $T$ to a space of the
type $S_{X_p}^{\,q}(\xi)$ in the case when $T$ is $p$-strongly
$q$-concave and $X$ is $p$-convex. To this end we will need to
describe the supremum on the right-hand side of the $p$-strongly
$q$-concave inequality in terms of the K\"{o}the dual of $X_p$.

\begin{lemma}\label{LEM: Sup-lr-Xp*}
If $X$ is $p$-convex and order semi-continuous then
$$
\sup_{(\beta_i)_{i\ge1}\in B_{\ell^r}}\Big\Vert\Big(\sum_{i=1}^n
|\beta_if_i|^p\Big)^{1/p}\,\Big\Vert_X=\sup_{h\in B_{(X_p)'}^+}
\Big(\sum_{i=1}^n\Big(\int|f_i|^ph\,d\mu\Big)^{q/p}\,\Big)^{1/q}
$$
for every finite subset $(f_i)_{i=1}^n\subset X$, where $1<
r\le\infty$ is such that $\frac{1}{r}=\frac{1}{p}-\frac{1}{q}$ and
$B_{(X_p)'}^+$ is the subset of all positive elements of the closed
unit ball $B_{(X_p)'}$ of $(X_p)'$.
\end{lemma}

\begin{proof}
Given $(f_i)_{i=1}^n\subset X$, since $X_p$ is order
semi-continuous, as $X$ is so, and $(\ell^{q/p})^*=\ell^{r/p}$, as
$\frac{r}{p}$ is the conjugate exponent of $\frac{q}{p}$, we have
that
\begin{eqnarray*}
\sup_{(\beta_i)\in B_{\ell^r}}\Big\Vert\Big(\sum_{i=1}^n
|\beta_if_i|^p\Big)^{1/p}\,\Big\Vert_X^p & = & \sup_{(\beta_i)\in
B_{\ell^r}}\Big\Vert\sum_{i=1}^n |\beta_if_i|^p\,\Big\Vert_{X_p}
\\ & = & \sup_{(\beta_i)\in
B_{\ell^r}}\sup_{h\in B_{(X_p)'}}\int
\sum_{i=1}^n|\beta_if_i|^p|h|\,d\mu \\ & = & \sup_{(\beta_i)\in
B_{\ell^r}}\sup_{h\in B_{(X_p)'}^+}\int
\sum_{i=1}^n|\beta_if_i|^ph\,d\mu \\ & = & \sup_{h\in
B_{(X_p)'}^+}\,\sup_{(\beta_i)\in B_{\ell^r}}
\sum_{i=1}^n|\beta_i|^p\int |f_i|^ph\,d\mu
\\ & = &
\sup_{h\in B_{(X_p)'}^+}\,\sup_{(\alpha_i)\in B_{\ell^{r/p}}^+}
\sum_{i=1}^n\alpha_i\int |f_i|^ph\,d\mu \\ & = & \sup_{h\in
B_{(X_p)'}^+}
\Big(\sum_{i=1}^n\Big(\int|f_i|^ph\,d\mu\Big)^{q/p}\,\Big)^{p/q}.
\end{eqnarray*}
\end{proof}

In the following remark, from Lemma \ref{LEM: Sup-lr-Xp*}, we obtain
easily an example of $p$-strongly $q$-concave operator.

\begin{remark}\label{REM: i-pq-concave}
Suppose that $X$ is $p$-convex and order semi-continuous. For every
finite positive Radon measure $\xi$ on $B_{(X_p)'}^+$ satisfying
\eqref{EQ: xiProperty}, it follows that the inclusion map $i\colon X
\to S_{X_p}^{\,q}(\xi)$ is $p$-strongly $q$-concave. Indeed, for
each $(f_i)_{i=1}^n \subset X$, we have that
\begin{eqnarray*}
\sum_{i=1}^n\Vert f_i\Vert_{S_{X_p}^{\,q}(\xi)}^q & = &
\sum_{i=1}^n\int_{B_{(X_p)'}^+}\Big(\int_\Omega|f_i(\omega)|^ph(\omega)\,d\mu(\omega)\Big)^{q/p}d\xi(h)
\\ & \le & \xi\big(B_{(X_p)'}^+\big)\sup_{h\in
B_{(X_p)'}^+}\sum_{i=1}^n\Big(\int_\Omega|f_i(\omega)|^ph(\omega)\,d\mu(\omega)\Big)^{q/p}
\end{eqnarray*}
and so, Lemma \ref{LEM: Sup-lr-Xp*} gives the conclusion for
$M_{p,q}(i)\le\xi\big(B_{(X_p)'}^+\big)^{1/q}$.
\end{remark}

Now let us prove our main result.

\begin{theorem}\label{THM: xiDomination}
If $T$ is $p$-strongly $q$-concave and $X$ is $p$-convex and order
semi-continuous, then there exists a probability Radon measure $\xi$
on $B_{(X_p)'}^+$ satisfying \eqref{EQ: xiProperty} such that
\begin{equation}\label{EQ: xiDomination}
\Vert T(f)\Vert_E\le
M_{p,q}(T)\Big(\int_{B_{(X_p)'}^+}\Big(\int_\Omega|f(\omega)|^ph(\omega)\,d\mu(\omega)\Big)^{q/p}\,d\xi(h)\Big)^{1/q}
\end{equation}
for all $f\in X$.
\end{theorem}

\begin{proof}
Recall that the stated topology on $(X_p)'$ is $\sigma((X_p)',X_p)$,
the one which is defined by the elements of $X_p$. For each finite
subset (with possibly repeated elements) $M=(f_i)_{i=1}^m\subset X$,
consider the map $\psi_M\colon B_{(X_p)'}^+\to [0,\infty)$ defined
by
$\psi_M(h)=\sum_{i=1}^m\big(\int_{\Omega}|f_i|^p\,h\,d\mu\big)^{q/p}$
for $h\in B_{(X_p)'}^+$. Note that $\psi_M$ attains its supremum as
it is continuous on a compact set, so there exists $h_M\in
B_{(X_p)'}^+$ such that $\sup_{h\in
B_{(X_p)'}^+}\psi_M(h)=\psi_M(h_M)$. Then, the $p$-strongly
$q$-concavity of $T$, together with Lemma \ref{LEM: Sup-lr-Xp*},
gives
\begin{eqnarray}\label{EQ: xM*-inequality}
\sum_{i=1}^m\Vert T(f_i)\Vert_E^q & \le & M_{p,q}(T)^q\sup_{h\in
B_{(X_p)'}^+}\sum_{i=1}^m\Big(\int_{\Omega}|f_i|^ph\,d\mu\Big)^{q/p}
\nonumber
\\ & \le & M_{p,q}(T)^q\sup_{h\in B_{(X_p)'}^+}\psi_M(h) \nonumber \\ & = &
M_{p,q}(T)^q\,\psi_M(h_M).
\end{eqnarray}
Consider now the continuous map $\phi_M\colon B_{(X_p)'}^+\to
\mathbb{R}$ defined by
$$
\phi_M(h)=M_{p,q}(T)^q\,\psi_M(h)-\sum_{i=1}^m\Vert T(f_i)\Vert_E^q
$$
for $h\in B_{(X_p)'}^+$. Take $B=\{\phi_M:\, M \textnormal{ is a
finite subset of } X\}$. Since for every $M=(f_i)_{i=1}^m,\,
M'=(f'_i)_{i=1}^k\subset X$ and $0<t<1$, it follows that
$t\phi_M+(1-t)\phi_{M'}=\phi_{M''}$ where
$M''=\big(t^{1/q}f_i\big)_{i=1}^m\cup\big((1-t)^{1/q}f'_i\big)_{i=1}^k$,
we have that $B$ is convex. Denote by $\mathcal{C}(B_{(X_p)'}^+)$
the space of continuous real functions on $B_{(X_p)'}^+$, endowed
with the supremum norm, and by $A$ the open convex subset $\{\phi\in
\mathcal{C}(B_{(X_p)'}^+):\, \phi(h)<0 \, \textnormal{ for all }
h\in B_{(X_p)'}^+\}$. By \eqref{EQ: xM*-inequality} we have that
$A\cap B=\emptyset$. From the Hahn-Banach separation theorem, there
exist $\xi\in \mathcal{C}(B_{(X_p)'}^+)^*$ and $\alpha\in\mathbb{R}$
such that $\langle\xi,\phi\rangle<\alpha\le\langle\xi,\phi_M\rangle$
for all $\phi\in A$ and $\phi_M\in B$. Since every negative constant
function is in $A$, it follows that $0 \le\alpha$. Even more,
$\alpha=0$ as the constant function equal to $0$ is just
$\phi_{\{0\}}\in B$. It is routine to see that
$\langle\xi,\phi\rangle\ge0$ whenever
$\phi\in\mathcal{C}(B_{(X_p)'}^+)$ is such that $\phi(h)\ge0$ for
all $h\in B_{(X_p)'}^+$. Then, $\xi$ is a positive linear functional
on $\mathcal{C}(B_{(X_p)'}^+)$ and so it can be interpreted as a
finite positive Radon measure on $B_{(X_p)'}^+$. Hence, we have that
$$
0\le\int_{B_{(X_p)'}^+}\phi_M\,d\xi
$$
for all finite subset $M\subset X$. Dividing by $\xi(B_{(X_p)'}^+)$,
we can suppose that $\xi$ is a probability measure. Then, for
$M=\{f\}$ with $f\in X$, we obtain that
$$
\Vert T(f)\Vert_E^q\le
M_{p,q}(T)^q\int_{B_{(X_p)'}^+}\Big(\int_{\Omega}|f(\omega)|^ph(\omega)\,d\mu(\omega)\Big)^{q/p}\,
d\xi(h)
$$
and so \eqref{EQ: xiDomination} holds.
\end{proof}

Actually, Theorem \ref{THM: xiDomination} says that we can find a
probability Radon measure $\xi$ on $B_{(X_p)'}^+$ such that $T\colon
X\to E$ is continuous when $X$ is considered with the norm of the
space $S_{X_p}^{\,q}(\xi)$. In the next result we will see how to
extend $T$ continuously to $S_{X_p}^{\,q}(\xi)$. Even more, we will
show that this extension is possible if and only if $T$ is
$p$-strongly $q$-concave.

\begin{theorem}\label{THM: SXpqExtension}
Suppose that $X$ is $p$-convex and order semi-continuous. The
following statements are equivalent:
\begin{itemize}\setlength{\leftskip}{-3ex}
\item[(a)] $T$ is $p$-strongly $q$-concave.

\item[(b)] There exists a probability Radon
measure $\xi$ on $B_{(X_p)'}^+$ satisfying \eqref{EQ: xiProperty}
such that $T$ can be extended continuously to $S_{X_p}^{\,q}(\xi)$,
i.e.\ there is a factorization for $T$ as
$$
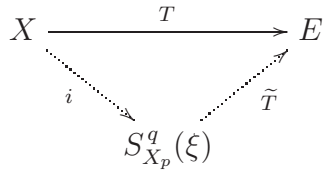
\xymatrix{
X \ar[rr]^T \ar@{.>}[dr]_(.4){i} & & E \\
& S_{X_p}^{\,q}(\xi) \ar@{.>}[ur]_(.6){\widetilde{T}} & }
$$
where $\widetilde{T}$ is a continuous linear operator and $i$ is the
inclusion map.
\end{itemize}
If (a)-(b) holds, then $M_{p,q}(T)=\Vert\widetilde{T}\Vert$.
\end{theorem}

\begin{proof}
(a) $\Rightarrow$ (b) From Theorem \ref{THM: xiDomination}, there is
a probability Radon measure $\xi$ on $B_{(X_p)'}^+$ satisfying
\eqref{EQ: xiProperty} such that $\Vert T(f)\Vert_E\le
M_{p,q}(T)\Vert f\Vert_{S_{X_p}^{\,q}(\xi)}$ for all $f\in X$. Given
$0\le f\in S_{X_p}^{\,q}(\xi)$, from Lemma \ref{LEM: saturatedBfs},
we can take $(f_n)_{n\ge1}\subset X$ such that $0\le f_n\uparrow f$
$\mu$-a.e. Then, since $S_{X_p}^{\,q}(\xi)$ is order continuous, we
have that $f_n\to f$ in $S_{X_p}^{\,q}(\xi)$ and so
$\big(T(f_n)\big)_{n\ge1}$ converges to some element $e$ of $E$.
Define $\widetilde{T}(f)=e$. Note that $\widetilde{T}$ is well
defined, since if $(g_n)_{n\ge1}\subset X$ is such that $0\le
g_n\uparrow f$ $\mu$-a.e., then
$$
\Vert T(f_n)-T(g_n)\Vert_E \le M_{p,q}(T)\Vert
f_n-g_n\Vert_{S_{X_p}^{\,q}(\xi)}\to0.
$$
Moreover,
\begin{eqnarray*}
\Vert \widetilde{T}(f)\Vert_E & = & \lim_{n\to\infty}\Vert
T(f_n)\Vert_E \\ & \le & M_{p,q}(T)\lim_{n\to\infty}\Vert
f_n\Vert_{S_{X_p}^{\,q}(\xi)} \\ & = & M_{p,q}(T)\Vert
f\Vert_{S_{X_p}^{\,q}(\xi)}.
\end{eqnarray*}
For a general $f\in S_{X_p}^{\,q}(\xi)$, writing $f=f^+-f^-$ where
$f^+$ and $f^-$ are the positive and negative parts of $f$
respectively, we define
$\widetilde{T}(f)=\widetilde{T}(f^+)-\widetilde{T}(f^-)$.  Then,
$\widetilde{T}\colon S_{X_p}^{\,q}(\xi)\to E$ is a continuous linear
operator extending $T$. Moreover $\Vert \widetilde{T}\Vert\le
M_{p,q}(T)$. Indeed, let $f\in S_{X_p}^{\,q}(\xi)$ and take
$(f_n^+)_{n\ge1},\,(f_n^-)_{n\ge1}\subset X$ such that $0\le
f_n^+\uparrow f^+$ and $0\le f_n^-\uparrow f^-$ $\mu$-a.e. Then,
$f_n^+-f_n^-\to f$ in $S_{X_p}^{\,q}(\xi)$ and
$$
T(f_n^+-f_n^-)=T(f_n^+)-T(f_n^-)\to
\widetilde{T}(f^+)-\widetilde{T}(f^-)=\widetilde{T}(f)
$$
in $E$. Hence,
\begin{eqnarray*}
\Vert \widetilde{T}(f)\Vert_E & = & \lim_{n\to\infty}\Vert
T(f_n^+-f_n^-)\Vert_E \\ & \le & M_{p,q}(T)\lim_{n\to\infty}\Vert
f_n^+-f_n^-\Vert_{S_{X_p}^{\,q}(\xi)} \\ & = & M_{p,q}(T)\Vert
f\Vert_{S_{X_p}^{\,q}(\xi)}.
\end{eqnarray*}

(b) $\Rightarrow$ (a) Given $(f_i)_{i=1}^n\subset X$, we have that
\begin{eqnarray*}
\sum_{i=1}^n\Vert T(f_i)\Vert_E^q & = & \sum_{i=1}^n\Vert
\widetilde{T}(f_i)\Vert_E^q \le
\Vert\widetilde{T}\Vert^q\sum_{i=1}^n\Vert
f_i\Vert_{S_{X_p}^{\,q}(\xi)}^q \\  & = & \Vert\widetilde{T}\Vert^q
\sum_{i=1}^n\int_{B_{(X_p)'}^+}\Big(\int_\Omega|f_i(\omega)|^ph(\omega)\,d\mu(\omega)\Big)^{q/p}d\xi(h)
\\ & \le & \Vert\widetilde{T}\Vert^q
\sup_{h\in
B_{(X_p)'}^+}\sum_{i=1}^n\Big(\int_\Omega|f_i(\omega)|^ph(\omega)\,d\mu(\omega)\Big)^{q/p}.
\end{eqnarray*}
That is, from Lemma \ref{LEM: Sup-lr-Xp*}, $T$ is $p$-strongly
$q$-concave with $M_{p,q}(T)\le\Vert\widetilde{T}\Vert$.
\end{proof}

A first  application of Theorem \ref{THM: SXpqExtension} is the
following Kakutani type representation theorem (see for instance
\cite[Theorem 1.b.2]{lindenstrauss-tzafriri}) for B.f.s.' being
order semi-continuous, $p$-convex and $p$-strongly $q$-concave.

\begin{corollary}\label{COR: i-isomorphism}
Suppose that $X$ is $p$-convex and order semi-continuous. The
following statements are equivalent:
\begin{itemize}\setlength{\leftskip}{-3ex}
\item[(a)] $X$ is $p$-strongly $q$-concave.

\item[(b)] There exists a probability Radon measure $\xi$ on
$B_{(X_p)'}^+$ satisfying \eqref{EQ: xiProperty}, such that
$X=S_{X_p}^{\,q}(\xi)$ with equivalent norms.
\end{itemize}
\end{corollary}

\begin{proof}
(a) $\Rightarrow$ (b) The identity map $I\colon X\to X$ is
$p$-strongly $q$-concave as $X$ is so. Then, from Theorem \ref{THM:
SXpqExtension}, there exists a probability Radon measure $\xi$ on
$B_{(X_p)'}^+$ satisfying \eqref{EQ: xiProperty}, such that $I$
factors as
$$
\xymatrix{
X \ar[rr]^I \ar@{.>}[dr]_(.4){i} & & X \\
& S_{X_p}^{\,q}(\xi) \ar@{.>}[ur]_(.6){\widetilde{I}} & }
$$
where $\widetilde{I}$ is a continuous linear operator with $\Vert
\widetilde{I}\Vert=M_{p,q}(X)$ and $i$ is the inclusion map. Since
$\xi$ is a probability measure, we have that $\Vert
f\Vert_{S_{X_p}^{\,q}(\xi)}\le\Vert f\Vert_X$ for all $f\in X$, see
the proof of Proposition \ref{PROP: SXpq(xi)-space}. Let $0\le f\in
S_{X_p}^{\,q}(\xi)$. By Lemma \ref{LEM: saturatedBfs}, we can take
$(f_n)_{n\ge1}\subset X$ such that $0\le f_n\uparrow f$ $\mu$-a.e.
Since $S_{X_p}^{\,q}(\xi)$ is order continuous, it follows that
$f_n\to f$ in $S_{X_p}^{\,q}(\xi)$ and so $f_n=\widetilde{I}(f_n)\to
\widetilde{I}(f)$ in $X$. Then, there is a subsequence of
$(f_n)_{n\ge1}$ converging $\mu$-a.e.\ to $\widetilde{I}(f)$ and
hence $f=\widetilde{I}(f)\in X$. For a general $f\in
S_{X_p}^{\,q}(\xi)$, writing $f=f^+-f^-$ where $f^+$ and $f^-$ are
the positive and negative parts of $f$ respectively, we have that
$f=\widetilde{I}(f^+)-\widetilde{I}(f^-)=\widetilde{I}(f)\in X$.
Therefore, $X=S_{X_p}^{\,q}(\xi)$ and $\widetilde{I}$ is de identity
map. Moreover, $\Vert f\Vert_X=\Vert \widetilde{I}(f)\Vert_X\le
\Vert \widetilde{I}\Vert\,\Vert
f\Vert_{S_{X_p}^{\,q}(\xi)}=M_{p,q}(X)\Vert
f\Vert_{S_{X_p}^{\,q}(\xi)}$ for all $f\in X$.

(b) $\Rightarrow$ (a) From Remark \ref{REM: i-pq-concave} it follows
that the identity map $I\colon X \to X$ is $p$-strongly $q$-concave.
\end{proof}

Note that under conditions of Corollary \ref{COR: i-isomorphism}, if
$X$ is $p$-strongly $q$-concave with constant $M_{p,q}(X)=1$, then
$X=S_{X_p}^{\,q}(\xi)$ with equal norms.


\section{$q$-summing operators on a $p$-convex B.f.s.}

Recall that a linear operator $T\colon X\to E$ between Banach spaces
is said to be \emph{$q$-summing} ($1\le q<\infty$) if there exists a
constant $C>0$ such that
$$
\Big(\sum_{i=1}^n\Vert Tx_i\Vert_E^q\Big)^{1/q}\le C\sup_{x^*\in
B_{X^*}}\Big(\sum_{i=1}^n|\langle x^*,x_i\rangle|^q\Big)^{1/q}
$$
for every finite subset $(x_i)_{i=1}^n\subset X$. Denote by
$\pi_q(T)$ the smallest possible value of $C$. Information about
$q$-summing operators can be found in \cite{diestel-jarchow-tonge}.

One of the main relations between summability and concavity for
operators defined on a B.f.s.\ $X$, is that every $q$-summing
operator is $q$-concave. This is a consequence of a direct
calculation which shows that for every $(f_i)_{i=1}^n\subset X$ and
$x^*\in X^*$ it follows that
\begin{equation}\label{EQ: q-norm}
\Big(\sum_{i=1}^n|\langle x^*,f_i\rangle|^q\Big)^{1/q}\le\Vert
x^*\Vert_{X^*}\Big\Vert\Big(\sum_{i=1}^n|f_i|^q\Big)^{1/q}\Big\Vert_X,
\end{equation}
see for instance \cite[Proposition 1.d.9]{lindenstrauss-tzafriri}
and the comments below. However, this calculation can be slightly
improved to obtain the following result.

\begin{proposition}\label{PROP: q-Summing}
Let $1\le p\le q<\infty$. Every $q$-summing linear operator $T\colon
X \to E$ from a B.f.s.\ $X$ into a Banach space $E$, is $p$-strongly
$q$-concave with $M_{p,q}(T)\le\pi_q(T)$.
\end{proposition}

\begin{proof}
Let $1<r\le\infty$ be such that
$\frac{1}{r}=\frac{1}{p}-\frac{1}{q}$ and consider a finite subset
$(f_i)_{i=1}^n\subset X$. We only have to prove
$$
\sup_{x^*\in B_{X^*}}\Big(\sum_{i=1}^n|\langle
x^*,f_i\rangle|^q\Big)^{1/q}\le\sup_{(\beta_i)_{i\ge1}\in
B_{\ell^r}}\Big\Vert\Big(\sum_{i=1}^n|\beta_if_i|^p\Big)^{1/p}\Big\Vert_X.
$$

Fix $x^*\in B_{X^*}$. Noting that $\frac{q}{p}$ and $\frac{r}{p}$
are conjugate exponents and using the inequality \eqref{EQ: q-norm},
we have
\begin{eqnarray*}
\Big(\sum_{i=1}^n|\langle x^*,f_i\rangle|^q\Big)^{1/q} & = &
\sup_{(\alpha_i)_{i\ge1}\in
B_{\ell^{r/p}}}\Big(\sum_{i=1}^n|\alpha_i||\langle
x^*,f_i\rangle|^p\Big)^{1/p} \\ & = & \sup_{(\beta_i)_{i\ge1}\in
B_{\ell^r}}\Big(\sum_{i=1}^n|\langle
x^*,\beta_if_i\rangle|^p\Big)^{1/p} \\ & \le &
\sup_{(\beta_i)_{i\ge1}\in B_{\ell^r}}\Big\Vert\Big(\sum_{i=1}^n
|\beta_if_i|^p\Big)^{1/p}\Big\Vert_X.
\end{eqnarray*}
Taking supremum in $x^*\in B_{X^*}$ we get the conclusion.
\end{proof}

From Proposition \ref{PROP: q-Summing}, Theorem \ref{THM:
SXpqExtension} and Remark \ref{REM: i-pq-concave}, we obtain the
final result.

\begin{corollary}\label{COR: q-summing-extension}
Set $1\le p\le q<\infty$. Let $X$ be a saturated order
semi-continuous $p$-convex B.f.s.\ and consider a $q$-summing linear
operator $T\colon X \to E$ with values in a Banach space $E$. Then,
there exists a probability Radon measure $\xi$ on $B_{(X_p)'}^+$
satisfying \eqref{EQ: xiProperty} such that $T$ can be factored as
$$
\xymatrix{
X \ar[rr]^T \ar@{.>}[dr]_(.4){i} & & E \\
& S_{X_p}^{\,q}(\xi) \ar@{.>}[ur]_(.6){\widetilde{T}} & }
$$
where $\widetilde{T}$ is a continuous linear operator with
$\Vert\widetilde{T}\Vert\le\pi_q(T)$ and $i$ is the inclusion map
which turns out to be $p$-strongly $q$-concave, and so $q$-concave.
\end{corollary}

Observe that what we obtain in Corollary \ref{COR:
q-summing-extension} is a proper extension for $T$, and not just a
factorization as the obtained in the Pietsch theorem for $q$-summing
operators through a subspace of an $L^q$-space.


\end{document}